\documentclass[ASNA]{USG}

\usepackage{mathtools}
\usepackage{mathrsfs}
\usepackage{bm}
\usepackage{graphicx}
\usepackage{subcaption}
\newcommand{\RR}{\mathbb{R}}

\newcommand{\Sym}{\operatorname{Sym}}

\newcommand{\Bn}{\mathcal{B}_{R}}
\newcommand{\LL}{\mathcal{L}}
\newcommand{\nn}{\boldsymbol{n}}
\newcommand{\x}{\boldsymbol{x}}
\newcommand{\y}{\boldsymbol{y}}
\newcommand{\bb}{\boldsymbol{b}}

\numberwithin{equation}{section}

\volume{0}          
\copyyear{2026}
\startpage{1}
\articledoi{}       

\begin{document}

\title{Symmetry Inheritance and Symmetry-Reduced Finite Element Analysis for Second-Order Linear Elliptic Problems with Robin Boundary Conditions}
\author[1]{Wei Jiang}
\author[1]{Xianlong Pan}

\authormark{PAN AND JIANG}
\titlemark{SYMMETRY OF ROBIN ELLIPTIC PROBLEMS}

\address[1]{\orgdiv{School of Physics and Mechatronics Engineering, }%
\orgname{Guizhou Minzu University, }%
\orgaddress{\city{Guiyang, }\postcode{550025, }\country{China}}}

\corres{Wei Jiang, School of Physics and Mechatronics Engineering, Guizhou Minzu University, Guiyang 550025, China; Email: \texttt{jwmathphy@163.com}}


\abstract[ABSTRACT]{A symmetry framework is developed for second-order linear elliptic equations subject to Robin boundary conditions on bounded domains. Orthogonal transformations of the domain are represented through left group actions on scalar, vector, and second-order tensor fields. The corresponding transformation rules for the gradient, divergence, diffusion flux, and conormal boundary term are derived in detail. Based on these relations, symmetry groups are introduced for the principal coefficient tensor, the first-order and zeroth-order coefficients, the differential operator, the Robin coefficient, and the boundary operator. The symmetry properties of the volume and boundary source terms are then incorporated into a common symmetry group for the complete boundary value problem. Under the assumption of unique solvability, it is proved that every element of this common group is also a symmetry of the solution. For reflection symmetries, homogeneous generalized Neumann conditions are obtained on artificial symmetry boundaries, leading to an exact reduction of the computational domain. A corresponding finite element formulation is presented, and numerical results are provided to verify the theoretical symmetry properties and the validity of the resulting domain-reduction strategy. Three examples illustrate radial reduction from multiple dimensions to one dimension, reflection-based domain reduction, and a variable-coefficient problem in which all coefficients and source terms are nonzero.}

\maketitle

\section{Introduction}

Second-order elliptic boundary value problems provide fundamental models in mathematical physics and engineering. In computational electromagnetics, scalar elliptic and Helmholtz-type equations arise in electrostatic and magnetostatic formulations, two-dimensional TE/TM reductions, waveguide models, and truncated-domain formulations with impedance-type boundary conditions; related finite element foundations may be found in \cite{Hiptmair2002,Monk2003}. Robin boundary conditions are equally classical in elliptic theory, where they describe a flux response coupled to the trace of the unknown and occur naturally in heat transfer, diffusion, acoustics, and impedance-type models. Their analytical treatment is well established in the theory of elliptic boundary value problems \cite{LionsMagenes1972,Daners2000,McLean2000}.

Symmetry methods for differential equations originate from the general theory of transformation groups and invariance \cite{Ovsiannikov1982,BlumanKumei1989,Olver1993}. The present work, however, concerns finite orthogonal transformations of the physical domain rather than Lie-symmetry reduction. For problems posed on geometrically symmetric domains, finite-group and representation-theoretic ideas have long been used to decompose field problems and to reduce the computational domain. Early developments include the use of finite-group representations in field computation \cite{Ballisti1982}, the systematic treatment of symmetric boundary value problems and finite element reduction by Bossavit \cite{Bossavit1986,Bossavit1993}, and applications to boundary element and three-dimensional eddy-current computations by Lobry and co-workers \cite{LobryBroche1994a,LobryBroche1994b,Lobry1996}. More recent work shows that symmetry remains relevant to modern finite element and computational electromagnetic formulations, including invariant bases in finite element exterior calculus \cite{Licht2024}, group-theoretic computation of waveguide modes \cite{Angiulli2020,GarciaContreras2021}, and group-representation-based finite element reduction for symmetric inhomogeneous waveguides \cite{Chu2024}.

A separate and extensive literature concerns the numerical approximation of elliptic problems with Robin boundary conditions. Classical isoparametric analysis already included Robin data on curved boundaries \cite{Lenoir1986}. Least-squares and mixed finite element formulations were subsequently developed for second-order elliptic Robin problems \cite{Lee1999,Konno2011}, while finite element well-posedness, regularity, and convergence have also been studied for generalized Robin conditions \cite{Kashiwabara2015}. Other developments include superconvergent recovery for Robin elliptic problems \cite{DuWuZhang2020}, isoparametric approximation on curved domains \cite{Edelmann2021}, unfitted $\phi$-FEM for natural Neumann and Robin conditions \cite{Duprez2023}, finite difference treatment of Robin data on irregular domains \cite{Chai2020}, and recent finite element analysis of generalized Robin problems on curved domains \cite{Kashiwabara2025}. Related numerical studies include finite element analysis for elliptic boundary value problems on nonsmooth domains \cite{Apel1996}, as well as the well-posedness and numerical treatment of variable-coefficient Helmholtz problems with Robin boundary conditions by means of boundary-domain integral equations \cite{FresnedaCaruso2025}.

These two lines of research address different aspects of the problem. Symmetry-based field computation is primarily concerned with group actions, invariant subspaces, and computational reduction, whereas most numerical studies of Robin problems focus on discretization, stability, regularity, and error analysis. For a Robin problem, however, a symmetry transformation must preserve not only the differential operator and the domain but also the conormal flux, the Robin coefficient, and the boundary source. Consequently, the boundary equation must be incorporated into the group action at the same level as the governing differential equation.

The present work extends the symmetry-inheritance framework previously developed for second-order linear elliptic Dirichlet boundary value problems \cite{PanJiang2026} to Robin boundary conditions. For an orthogonal transformation, left group actions are introduced consistently on scalar, vector, and second-order tensor fields. The transformation properties of the gradient, divergence, diffusion flux, and conormal term are then used to define the symmetry groups of the principal coefficient tensor, the first-order and zeroth-order coefficients, the differential operator, the Robin coefficient, and the Robin boundary operator. By including the volume source and the boundary source, a common symmetry group is obtained for the complete boundary value problem.

The remainder of the paper is organized as follows. Section~2 introduces the Robin boundary value problem and the orthogonal group actions. Section~3 defines the coefficient and operator symmetry groups and proves the symmetry-inheritance theorem. Section~4 derives the reflection-induced artificial boundary condition, the reduced-domain problem and the finite element discretization. Section~5 presents three representative examples. Conclusions and possible extensions are summarized in the final section.

\section{The Robin Boundary Value Problem and Group Actions}\label{sec:BVPaction}
Let $\Omega\subset\RR^n$ be a bounded Lipschitz domain with boundary $\partial\Omega$. The following second-order linear elliptic boundary value problem with a Robin boundary condition is considered:
\begin{subequations}\label{eq:fullBVP}
\begin{align}
\LL u:=-\nabla\cdot\bigl(A(\x)\nabla u(\x)\bigr)
+\widehat{\bb}(\x)\cdot\nabla u(\x)+c(\x)u(\x)&=f(\x),
&&\x\in\Omega,\label{eq:main}\\
\Bn u:=\bigl(A(\x)\nabla u(\x)\bigr)\cdot \nn(\x)+\beta(\x)u(\x)&=h(\x),
&&\x\in\partial\Omega.\label{eq:robin}
\end{align}
\end{subequations}
Here, $A:\Omega\to\RR^{n\times n}$ is symmetric and uniformly positive definite, $\widehat{\bb}:\Omega\to\RR^n$ is a prescribed vector-valued coefficient, and $c,f:\Omega\to\RR$ are scalar-valued functions. The functions $\beta,h:\partial\Omega\to\RR$ are prescribed boundary data, and $\nn$ denotes the outward unit normal on $\partial\Omega$. Sufficient regularity is assumed whenever pointwise differential identities are used, and the boundary value problem is assumed to be uniquely solvable when symmetry inheritance is invoked.

Problems of the form \eqref{eq:fullBVP} arise in diffusion, heat transfer, mass transport, and wave propagation. Robin boundary conditions model partial exchange or impedance-type interactions and also occur in scalar formulations of computational electromagnetics involving impedance or absorbing boundaries.

Only orthogonal spatial symmetries are needed below. After choosing the origin at a common fixed point of the bounded symmetric domain, define
\begin{equation}\label{eq:GOmega}
G_\Omega:=\{g\in O(n):g\Omega=\Omega\},
\end{equation}
where $O(n)$ denotes the group of $n\times n$ orthogonal matrices. In fact, $G_\Omega$ is called the isometry group of the bounded domain $\Omega$.

For a scalar field $v$, a polar vector field $\boldsymbol v$, and a second-order tensor field $M$, respectively, define
\begin{align}
(T_g^{(0)}v)(\x)&:=v(g^{-1}\x),\label{eq:leftscalar}\\
(T_g^{(1)}\boldsymbol v)(\x)&:=g\,\boldsymbol v(g^{-1}\x),\label{eq:leftvector}\\
(T_g^{(2)}M)(\x)&:=g\,M(g^{-1}\x)g^{-1}.\label{eq:lefttensor}
\end{align}
The same notation $T_g^{(0)}$ is used for scalar functions on $\partial\Omega$, since $g\partial\Omega=\partial\Omega$ for every $g\in G_\Omega$. These definitions give left actions. For example,
\begin{align*}
(T_{g_1}^{(0)}T_{g_2}^{(0)}v)(\x)
&=v(g_2^{-1}g_1^{-1}\x)
=v((g_1g_2)^{-1}\x)
=(T_{g_1g_2}^{(0)}v)(\x),
\end{align*}
and the same composition law holds for $T_g^{(1)}$ and $T_g^{(2)}$.

The relations between these actions and the differential operators are needed before any symmetry group is introduced. Let $\y=g^{-1}\x$ and $g=(g_{ij})_{n\times n}\in O(n)$. Since $g^{-1}=g^T$, one has $y_j=\sum_{k=1}^n g_{kj}x_k$ and therefore $\partial y_j/\partial x_i=g_{ij}$. For a differentiable scalar function $v$,
\begin{align}
\frac{\partial}{\partial x_i}(T_g^{(0)}v)(\x)
&=\frac{\partial}{\partial x_i}v(\y)
=\sum_{j=1}^n\frac{\partial v}{\partial y_j}(\y)
\frac{\partial y_j}{\partial x_i}=\sum_{j=1}^n g_{ij}\frac{\partial v}{\partial y_j}(\y).
\end{align}
Hence
\begin{equation}\label{eq:gradienttransform}
\nabla T_g^{(0)}v=T_g^{(1)}\nabla v.
\end{equation}
Thus the gradient of a scalar field transforms as a polar vector field.

Let $\boldsymbol F=(F_1,\ldots,F_n)^T$ be differentiable. Using \eqref{eq:leftvector} and the chain rule gives
\begin{align}
\nabla\cdot(T_g^{(1)}\boldsymbol F)(\x)
&=\sum_{i=1}^n\frac{\partial}{\partial x_i}
\left(\sum_{j=1}^n g_{ij}F_j(\y)\right)=\sum_{i,j,k}g_{ij}(g^{-1})_{ki}
\frac{\partial F_j}{\partial y_k}(\y)\nonumber\\
&=\sum_{j,k}\delta_{jk}\frac{\partial F_j}{\partial y_k}(\y)
=(\nabla\cdot\boldsymbol F)(g^{-1}\x).
\end{align}
Therefore
\begin{equation}\label{eq:divtransform}
\nabla\cdot T_g^{(1)}\boldsymbol F
=T_g^{(0)}(\nabla\cdot\boldsymbol F).
\end{equation}

The tensor and vector actions are compatible with matrix--vector multiplication:
\begin{align}
(T_g^{(2)}A)(T_g^{(1)}\boldsymbol v)
&=gA(g^{-1}\x)g^{-1}g\boldsymbol v(g^{-1}\x)\nonumber\\
&=g[A(g^{-1}\x)\boldsymbol v(g^{-1}\x)]
=T_g^{(1)}(A\boldsymbol v).
\end{align}
Consequently,
\begin{equation}\label{eq:tensormult}
(T_g^{(2)}A)(T_g^{(1)}\boldsymbol v)=T_g^{(1)}(A\boldsymbol v).
\end{equation}
Because $g$ is orthogonal, the Euclidean inner product is preserved, so
\begin{equation}\label{eq:dottransform}
(T_g^{(1)}\boldsymbol p)\cdot(T_g^{(1)}\boldsymbol q)
=T_g^{(0)}(\boldsymbol p\cdot\boldsymbol q).
\end{equation}
Combining \eqref{eq:gradienttransform}, \eqref{eq:divtransform}, and \eqref{eq:tensormult} gives
\begin{equation}\label{eq:principaltransform}
T_g^{(0)}\!\left[\nabla\cdot(A\nabla u)\right]
=\nabla\cdot\left[(T_g^{(2)}A)\nabla(T_g^{(0)}u)\right].
\end{equation}
Similarly, from \eqref{eq:dottransform} and \eqref{eq:gradienttransform}, then
\begin{equation}\label{eq:firstordertransform}
T_g^{(0)}\!\left(\widehat{\bb}\cdot\nabla u\right)=(T_g^{(1)}\widehat{\bb})\cdot (T_g^{(1)}\nabla u)
=(T_g^{(1)}\widehat{\bb})\cdot\nabla(T_g^{(0)}u),
\end{equation}
and
\begin{equation}\label{eq:zerothtransform}
T_g^{(0)}(cu)=(T_g^{(0)}c)(T_g^{(0)}u).
\end{equation}
are valid. Thus, application of $T_g^{(0)}$ to both sides of the governing equation \eqref{eq:main} yields
\begin{align}
-\nabla\cdot\left[(T_g^{(2)}A)\nabla(T_g^{(0)}u)\right]
+(T_g^{(1)}\widehat{\bb})\cdot\nabla(T_g^{(0)}u)+(T_g^{(0)}c)(T_g^{(0)}u)=T_g^{(0)}f.
\label{eq:transformedPDE}
\end{align}

The Robin boundary equation transforms under the same action. Since $g\Omega=\Omega$ and $g$ is orthogonal, the tangent space at $\x\in\partial\Omega$ is mapped onto the tangent space at $g\x$, and the outward orientation is preserved. Hence
\begin{equation}\label{eq:normaltransform}
\nn(g\x)=g\nn(\x),\qquad \x\in\partial\Omega.
\end{equation}
Equivalently,
\begin{equation}\label{eq:normalfixed}
T_g^{(1)}\nn=\nn\qquad\text{on }\partial\Omega.
\end{equation}
The conormal term is a scalar boundary function. Applying $T_g^{(0)}$ directly and then using \eqref{eq:dottransform}, \eqref{eq:tensormult}, and \eqref{eq:gradienttransform} gives
\begin{align}
T_g^{(0)}\left[(A\nabla u)\cdot\nn\right]
&=T_g^{(1)}(A\nabla u)\cdot T_g^{(1)}\nn\nonumber\\
&=\left[(T_g^{(2)}A)T_g^{(1)}(\nabla u)\right]\cdot T_g^{(1)}\nn\nonumber\\
&=\left[(T_g^{(2)}A)\nabla(T_g^{(0)}u)\right]\cdot\nn.
\label{eq:conormaltransform}
\end{align}
For completeness, the first equality in \eqref{eq:conormaltransform} follows pointwise from
\begin{align}
&T_g^{(0)}\left[(A\nabla u)\cdot\nn\right](\x)=A(g^{-1}\x)\nabla u(g^{-1}\x)\cdot\nn(g^{-1}\x)\nonumber\\
&=\bigl[gA(g^{-1}\x)\nabla u(g^{-1}\x)\bigr]\cdot\bigl[g\nn(g^{-1}\x)\bigr]=\bigl[gA(g^{-1}\x)g^{-1}][g\nabla u(g^{-1}\x)\bigr]\cdot\bigl[g\nn(g^{-1}\x)\bigr],
\end{align}
where orthogonality of $g$ is used in the last step. The remaining Robin term satisfies
\begin{equation}\label{eq:betatransform}
T_g^{(0)}(\beta u)=(T_g^{(0)}\beta)(T_g^{(0)}u)
\qquad\text{on }\partial\Omega.
\end{equation}
Therefore application of $T_g^{(0)}$ to both sides of \eqref{eq:robin} gives
\begin{equation}\label{eq:transformedRobin}
\left[(T_g^{(2)}A)\nabla(T_g^{(0)}u)\right]\cdot\nn
+(T_g^{(0)}\beta)(T_g^{(0)}u)=T_g^{(0)}h
\qquad\text{on }\partial\Omega.
\end{equation}
Equations \eqref{eq:transformedPDE} and \eqref{eq:transformedRobin} show explicitly how the same left group action acts on both equations of the complete Robin boundary value problem.

\section{Symmetry Groups and Symmetry Inheritance}\label{sec:groups}
The preceding transformation formulas indicate which data must remain unchanged for the transformed problem to coincide with the original one. This provides a natural route to the required symmetry groups.

For the principal matrix-valued coefficient, define
\begin{equation}\label{eq:SymA}
\Sym(A):=\{g\in G_\Omega:T_g^{(2)}A=A\}.
\end{equation}
Equivalently, $g\in\Sym(A)$ if and only if
\begin{equation}\label{eq:SymApoint}
A(g\x)=gA(\x)g^{-1}=gA(\x)g^T,
\qquad \x\in\Omega.
\end{equation}
For the first-order vector coefficient, define
\begin{equation}\label{eq:Symb}
\Sym(\widehat{\bb}):=\{g\in G_\Omega:T_g^{(1)}\widehat{\bb}=\widehat{\bb}\},
\end{equation}
which is equivalent to
\begin{equation}\label{eq:Symbpoint}
\widehat{\bb}(g\x)=g\widehat{\bb}(\x),
\qquad \x\in\Omega.
\end{equation}
Since $c$ is scalar, its symmetry group is
\begin{equation}\label{eq:Symc}
\Sym(c):=\{g\in G_\Omega:T_g^{(0)}c=c\}
=\{g\in G_\Omega:c(g\x)=c(\x),~\forall \x\in{\Omega}\}.
\end{equation}
These definitions are the same coefficient-function symmetry groups used in \cite{PanJiang2026}, now expressed directly in terms of the left actions \eqref{eq:leftscalar}--\eqref{eq:lefttensor}.

The differential operator \(\LL\) contains precisely the three coefficients $A$, $\widehat{\bb}$, and $c$. Its symmetry group is therefore defined by their common stabilizer,
\begin{equation}\label{eq:SymL}
\Sym(\LL):=\Sym(A)\cap\Sym(\widehat{\bb})\cap\Sym(c).
\end{equation}
This definition is consistent with the operator transformation law. Indeed, if $g\in\Sym(\LL)$, then $T_g^{(2)}A=A$, $T_g^{(1)}\widehat{\bb}=\widehat{\bb}$, and $T_g^{(0)}c=c$. When $g\in{\Sym(\LL)}$, equation \eqref{eq:transformedPDE} then reduces to
\begin{equation}\label{eq:Lintertwine}
\LL(T_g^{(0)}u)=T_g^{(0)}(\LL u).
\end{equation}
Thus $\Sym(\LL)$ is a group of transformations under which the differential operator intertwines with the left action on scalar functions.

The Robin coefficient $\beta$ is a scalar function defined on $\partial\Omega$. Since $g$ maps $\partial\Omega$ onto itself, its symmetry group is naturally defined by
\begin{equation}\label{eq:Symbeta}
\Sym(\beta):=\{g\in G_\Omega:T_g^{(0)}\beta=\beta\text{ on }\partial\Omega\}.
\end{equation}
Equivalently, $\beta(g\x)=\beta(\x)$ for every $\x\in\partial\Omega$. The Robin boundary operator depends on $A$ through the conormal flux and on $\beta$ through the zeroth-order boundary term. Accordingly, define
\begin{equation}\label{eq:SymB}
\Sym(\Bn):=\Sym(A)\cap\Sym(\beta).
\end{equation}
This definition has the required operator meaning. If $g\in\Sym(\Bn)$, then \eqref{eq:conormaltransform} and \eqref{eq:betatransform} give
\begin{align}
T_g^{(0)}(\Bn v)
&=T_g^{(0)}\left[(A\nabla v)\cdot\nn+\beta v\right]\nonumber\\
&=\left[A\nabla(T_g^{(0)}v)\right]\cdot\nn
+\beta T_g^{(0)}v\nonumber\\
&=\Bn(T_g^{(0)}v),
\qquad \text{on }\partial\Omega.
\label{eq:Bintertwine}
\end{align}
Hence membership in $\Sym(A)\cap\Sym(\beta)$ guarantees covariance of the complete Robin boundary operator. The normal field does not require an additional symmetry group because its transformation law \eqref{eq:normaltransform} follows from $g\in G_\Omega$ and the geometry of the boundary.

The two source terms are treated by the same scalar action. For the interior source,
\begin{equation}\label{eq:Symf}
\Sym(f):=\{g\in G_\Omega:T_g^{(0)}f=f\text{ in }\Omega\},
\end{equation}
and for the boundary source,
\begin{equation}\label{eq:Symh}
\Sym(h):=\{g\in G_\Omega:T_g^{(0)}h=h\text{ on }\partial\Omega\}.
\end{equation}
All the sets in \eqref{eq:SymA}, \eqref{eq:Symb}, \eqref{eq:Symc}, \eqref{eq:Symbeta}, \eqref{eq:Symf}, and \eqref{eq:Symh} are stabilizers of left group actions and are therefore subgroups of $G_\Omega$. Their intersections, including \eqref{eq:SymL} and \eqref{eq:SymB}, are also subgroups.

The complete boundary value problem \eqref{eq:fullBVP} is unchanged only when every coefficient and both source terms are simultaneously invariant. This leads to the common symmetry group
\begin{align}
G_{\rm com}
&:=\Sym(A)\cap\Sym(\widehat{\bb})\cap\Sym(c)\cap\Sym(f)
\cap\Sym(\beta)\cap\Sym(h)\label{eq:Gcom}\\
&=\Sym(\LL)\cap\Sym(f)\cap\Sym(\Bn)\cap\Sym(h).\nonumber
\end{align}
The first line emphasizes the symmetry of all data $A,\widehat{\bb},c,f,\beta,h$, whereas the second line emphasizes the decomposition into the interior equation and the Robin boundary equation.

The symmetry group of the unique solution $u$ to equation \eqref{eq:fullBVP} is defined by
\begin{equation}\label{eq:Symu}
\Sym(u):=\{g\in G_\Omega:T_g^{(0)}u=u\}.
\end{equation}
The relation between $G_{\rm com}$ and $\Sym(u)$ is the central consequence of the preceding definitions.

\begin{theorem}[Symmetry inheritance]\label{thm:solution_symmetry}
Assume that the Robin boundary value problem \eqref{eq:fullBVP} has a unique solution $u$. Then
\begin{equation}\label{eq:maininclusion}
G_{\rm com}\subseteq\Sym(u).
\end{equation}
\end{theorem}

\begin{proof}
Let $g\in G_{\rm com}$ be arbitrary and set $u_g=T_g^{(0)}u$. By the definition of $G_{\rm com}$, one has $g\in\Sym(\LL)$ and $g\in\Sym(f)$. Hence the operator relation \eqref{eq:Lintertwine} gives
\begin{align}
\LL u_g
&=\LL(T_g^{(0)}u)
=T_g^{(0)}(\LL u)
=T_g^{(0)}f
=f
\qquad\text{in }\Omega.
\label{eq:inheritinterior}
\end{align}
Thus the transformed function satisfies the same governing differential equation as $u$.

Similarly, one has $g\in\Sym(\Bn)$ and $g\in\Sym(h)$. Using \eqref{eq:Bintertwine},
\begin{align}
\Bn u_g
&=\Bn(T_g^{(0)}u)
=T_g^{(0)}(\Bn u)
=T_g^{(0)}h
=h
\qquad\text{on }\partial\Omega.
\label{eq:inheritboundary}
\end{align}
Therefore $u_g$ satisfies the same Robin boundary condition as $u$. Combining \eqref{eq:inheritinterior} and \eqref{eq:inheritboundary},
\begin{equation}\label{eq:sameproblem}
\begin{cases}
\LL u_g=f,&\text{in }\Omega,\\
\Bn u_g=h,&\text{on }\partial\Omega.
\end{cases}
\end{equation}
Hence $u_g$ and $u$ solve exactly the same boundary value problem. Uniqueness implies $u_g=u$, or equivalently $T_g^{(0)}u=u$. Thus $g\in\Sym(u)$. Since $g$ is arbitrary, \eqref{eq:maininclusion} follows.
\end{proof}

Note that the inclusion in Theorem~\ref{thm:solution_symmetry} may be strict. Consider the following Robin boundary problem:
\begin{equation}
\begin{cases}
-\nabla\cdot(A\nabla u)+u=1, & \text{in }\Omega,\\
(A\nabla u)\cdot\nn+u=1, & \text{on }\partial\Omega,
\end{cases}
\qquad
\Omega=(-1,1)^2,\qquad
A=\begin{pmatrix}2&0\\0&1\end{pmatrix}.
\end{equation}
Here $\widehat{\bb}={\bf 0}$ and $c=\beta=f=h=1$. The orthogonal symmetry group of the square is the dihedral group $D_4$, consisting of its four rotations and four reflections. Let
\[
Q_x=\begin{pmatrix}-1&0\\0&1\end{pmatrix},
\qquad
Q_y=\begin{pmatrix}1&0\\0&-1\end{pmatrix},
\]
which represent the reflections $(x,y)\mapsto(-x,y)$ and $(x,y)\mapsto(x,-y)$, respectively.

Since the two eigenvalues of $A$ are distinct, only the symmetries preserving the two coordinate axes leave $A$ invariant. Hence
\[
\Sym(A)=\Sym(\LL)=\Sym(\Bn)=\{I_2,Q_x,Q_y,-I_2\},
\]
whereas
\[
\Sym(\widehat{\bb})=\Sym(c)=\Sym(\beta)=\Sym(f)=\Sym(h)=D_4.
\]
Therefore
\[
G_{\rm com}=\{I_2,Q_x,Q_y,-I_2\}.
\]

The problem is uniquely solvable by coercivity, and its exact solution is $u\equiv1$. Since a constant function is invariant under every element of $D_4$,
\[
\Sym(u)=D_4,
\]
and consequently
\[
G_{\rm com}=\{I_2,Q_x,Q_y,-I_2\}\subsetneq D_4=\Sym(u).
\]
Thus, the common symmetry group gives the symmetries necessarily inherited from the problem data, while a particular solution may possess additional symmetries.

\begin{remark}
In general, $G_{\rm com}\subseteq\Sym(u)$ need not be an equality. Additional symmetries may arise from the particular structure of the solution.
\end{remark}

\section{Domain-Reduced Finite Element Methdod}\label{sec:reduction}
Assume that $Q\in G_{\rm com}$ is the reflection with respect to the hyperplane
\begin{equation}\label{eq:reflectionplane}
\Pi=\{\x\in\RR^n:\boldsymbol n_S\cdot\x=0\},
\end{equation}
where $\boldsymbol n_S$ is a unit normal and
\begin{equation}\label{eq:reflection}
Q=I_n-2\boldsymbol n_S\boldsymbol n_S^T.
\end{equation}
Theorem~\ref{thm:solution_symmetry} gives $u(Q\x)=u(\x)$. At a point $\x\in\Pi\cap\Omega$, the identity $Q\x=\x$ and the gradient transformation law imply
\begin{equation}\label{eq:gradientfixedplane}
Q\nabla u(\x)=\nabla u(\x).
\end{equation}
Thus $\nabla u$ is tangent to the reflection plane and $\nabla u\cdot\boldsymbol n_S=0$ there. For an anisotropic coefficient, the exact artificial boundary condition concerns the conormal flux. Since $Q\in\Sym(A)$ and $\Sym(A)$ is a group, then $Q^{-1}\in{\Sym(A)}$. This implies 
\begin{equation}\label{eq:Acommute}
Q^{-1}A(Q\x)Q=A(\x),~\forall\x\in{\Omega}
\end{equation}
Particularly, when $\x\in{\Pi}$, it follows that $Q\x=\x$ on $\Pi$, then
yields $QA=AQ$ on the fixed plane. Applying $Q$ to the flux and using \eqref{eq:gradientfixedplane},
\begin{equation}
Q(A\nabla u)=A(Q\nabla u)=A\nabla u,~\x\in{\Pi}
\end{equation}
Hence $A\nabla u$ is also tangent to $\Pi$, and therefore
\begin{equation}\label{eq:generalized_neumann}
(A\nabla u)\cdot\boldsymbol n_S=0
\qquad\text{on }\Pi\cap\Omega.
\end{equation}
The derivation is identical in structure to the Dirichlet case in \cite{PanJiang2026}: the artificial condition is determined by the reflection symmetry of the solution and the tensor symmetry of $A$, not by the type of boundary condition imposed on the original exterior boundary.

Suppose that reflections contained in $G_{\rm com}$ partition $\Omega$ into congruent fundamental subdomains, and let $\Omega_1$ denote one of them. Write
\begin{equation}\label{eq:boundarysplit}
\partial\Omega_1=\Gamma_R\cup\Gamma_S,
\end{equation}
where $\Gamma_R=\partial\Omega_1\cap\partial\Omega$ is inherited from the original Robin boundary and $\Gamma_S$ consists of artificial symmetry cuts. The reduced problem is
\begin{subequations}\label{eq:reduced}
\begin{align}
-\nabla\cdot(A\nabla u_1)+\widehat{\bb}\cdot\nabla u_1+cu_1&=f,
&&\text{in }\Omega_1,\label{eq:reduceda}\\
(A\nabla u_1)\cdot\nn+\beta u_1&=h,
&&\text{on }\Gamma_R,\label{eq:reducedb}\\
(A\nabla u_1)\cdot\nn_S&=0,
&&\text{on }\Gamma_S.\label{eq:reducedc}
\end{align}
\end{subequations}
Here $\nn$ is the outward unit normal of the original Robin boundary $\Gamma_R$, whereas $\nn_S$ is the outward unit normal of artificial symmetry boundaries $\Gamma_S$. The restriction $u_1=u|_{\Omega_1}$ satisfies \eqref{eq:reduced}, and the full-domain solution is reconstructed by the corresponding symmetry transformations.

Let $V=H^1(\Omega_1)$. Multiplying \eqref{eq:reduceda} by a test function $v\in V$, integrating the principal term by parts, using the Robin condition on $\Gamma_R$, and using the homogeneous generalized Neumann condition on $\Gamma_S$, gives
\begin{equation}\label{eq:reducedweak}
a_1(u_1,v)=F_1(v)\qquad\forall v\in V,
\end{equation}
where
\begin{equation}\label{eq:a1}
\begin{split}
a_1(w,v)=&\int_{\Omega_1}(A\nabla w)\cdot\nabla v\,d\x
+\int_{\Omega_1}(\widehat{\bb}\cdot\nabla w)v\,d\x+\int_{\Omega_1}cwv\,d\x
+\int_{\Gamma_R}\beta wv\,ds,
\end{split}
\end{equation}
and
\begin{equation}\label{eq:F1}
F_1(v)=\int_{\Omega_1}fv\,d\x+\int_{\Gamma_R}hv\,ds.
\end{equation}
No boundary integral remains on $\Gamma_S$ because its conormal flux is zero.

Let $\mathcal T_h$ be a conforming simplicial mesh of $\Omega_1$, and let
\begin{equation}\label{eq:Vh}
V_h=\{v_h\in C^0(\overline{\Omega}_1):v_h|_K\text{ is affine for every }K\in\mathcal T_h\}.
\end{equation}
The linear finite element approximation seeks $u_h\in V_h$ such that
\begin{equation}\label{eq:fem}
a_1(u_h,v_h)=F_1(v_h)\qquad\forall v_h\in V_h.
\end{equation}
With the nodal basis $\{\varphi_j\}_{j=1}^{N_h}$ and $u_h=\sum_jU_j\varphi_j$, the algebraic problem is
\begin{equation}\label{eq:system}
KU=F,
\end{equation}
with
\begin{equation}\label{eq:matrix}
\begin{split}
K_{ij}={}&\int_{\Omega_1}(A\nabla\varphi_j)\cdot\nabla\varphi_i\,d\x
+\int_{\Omega_1}(\widehat{\bb}\cdot\nabla\varphi_j)\varphi_i\,d\x\\
&+\int_{\Omega_1}c\varphi_j\varphi_i\,d\x
+\int_{\Gamma_R}\beta\varphi_j\varphi_i\,ds,
\end{split}
\end{equation}
and
\begin{equation}\label{eq:rhs}
F_i=\int_{\Omega_1}f\varphi_i\,d\x+\int_{\Gamma_R}h\varphi_i\,ds.
\end{equation}
Thus the Robin coefficient contributes a boundary matrix and the boundary source contributes a boundary load, whereas the artificial symmetry boundary requires no additional term.

\section{Examples}\label{sec:examples}
 Three examples are presented to illustrate dimensional reduction, analytical verification of the symmetry theorem, and numerical validation of the reduced-domain formulation.
\paragraph{Example 1: exact reduction from $n$ dimensions to one dimension.}\leavevmode\par
\noindent Let
\begin{equation}
\Omega=\{\x\in\RR^n:R_1<\|\x\|<R_2\}
\end{equation}
be a spherical annulus, and consider
\begin{subequations}\label{eq:radial_general}
\begin{align}
-\nabla\cdot\!\left(\varepsilon(\|\x\|)\nabla u\right)
+c(\|\x\|)u
&=f(\|\x\|),
&&\x\in\Omega, \label{eq:radial_general_a}\\
\varepsilon(\|\x\|)\nabla u\cdot\nn
+\beta(\|\x\|)u
&=h(\|\x\|),
&&\x\in\partial\Omega. \label{eq:radial_general_b}
\end{align}
\end{subequations}
Assume that the problem is uniquely solvable and that
$\varepsilon(\|\x\|)>0$. Since the domain, coefficients, and source
data are invariant under every orthogonal transformation, then $G_{\rm com}=O(n)$ holds. Theorem~\ref{thm:solution_symmetry} therefore gives $O(n)\subseteq\Sym(u)$.
Hence $u$ is constant on every orbit of the action of $O(n)$, i.e.,
on every sphere centered at the origin. Consequently,
\begin{equation}
u(\x)=\phi(r),\qquad r=\|\x\|.
\end{equation}
Thus the quotient of the spherical annulus by the $O(n)$ action is
the interval $(R_1,R_2)$, and the original $n$-dimensional problem is
reduced exactly to
\begin{subequations}\label{eq:radial_ode}
\begin{align}
-\frac{1}{r^{n-1}}\frac{d}{dr}
\left(r^{n-1}\varepsilon(r)\phi'(r)\right)
+c(r)\phi(r)
&=f(r),
&&R_1<r<R_2, \label{eq:radial_ode_a}\\
-\varepsilon(R_1)\phi'(R_1)+\beta_1\phi(R_1)
&=h_1, \label{eq:radial_ode_b}\\
\varepsilon(R_2)\phi'(R_2)+\beta_2\phi(R_2)
&=h_2, \label{eq:radial_ode_c}
\end{align}
\end{subequations}
where $\beta_i=\beta(R_i)$ and $h_i=h(R_i)$ for $i=1,2$.

To illustrate the reduction independently, consider the
three-dimensional case $n=3$ with $R_1=1$, $R_2=2$,
$\varepsilon\equiv1$, $c\equiv1$, and $\beta\equiv2$:
\begin{subequations}\label{eq:radial_3d}
\begin{align}
-\Delta u+u
&=\|\x\|^2-5,
&&1<\|\x\|<2, \label{eq:radial_3d_a}\\
\frac{\partial u}{\partial n}+2u
&=
\begin{cases}
2, & \|\x\|=1,\\
14, & \|\x\|=2.
\end{cases}
\label{eq:radial_3d_b}
\end{align}
\end{subequations}
A direct substitution shows that
\begin{equation}
u_{\rm ex}(\x)=1+\|\x\|^2
\end{equation}
is the exact solution. 

All coefficients and source data of this three-dimensional problem are
invariant under $O(3)$, so that $G_{\rm com}=O(3)$. Independently, the explicit solution depends only on $\|\x\|$, and hence $\Sym(u_{\rm ex})=O(3)$. Therefore $G_{\rm com}=\Sym(u_{\rm ex})=O(3)$, which is consistent with Theorem~\ref{thm:solution_symmetry}.

After reduction, the exact radial function $\phi(r)=1+r^2$ 
satisfies
\begin{subequations}
\begin{align}
-\frac{1}{r^2}\frac{d}{dr}
\left(r^2\phi'(r)\right)+\phi(r)
&=r^2-5,
&&1<r<2,\\
-\phi'(1)+2\phi(1)&=2,\\
\phi'(2)+2\phi(2)&=14.
\end{align}
\end{subequations}
Hence both the original three-dimensional problem and the reduced
one-dimensional problem have the same analytically known solution under
the radial identification $u(\x)=\phi(\|\x\|)$. This example illustrates
that the symmetry theorem can lead to an exact reduction of spatial
dimension, rather than merely a reduction to a smaller subdomain.

\paragraph{Example 2: a manufactured reflection-symmetric Robin problem.}\leavevmode\par
\noindent Let $\Omega=(-1,1)^2$, $A=I_2$, $\widehat{\bb}={\bf 0}$, $c=1$, and $\beta=2$. Set $f(x,y)=(5\pi^2+1)\cos(\pi x)\cos(2\pi y)$ and
\begin{equation*}
h(x,y)=
\begin{cases}
2\cos(\pi x), & y=\pm1,\\
-2\cos(2\pi y), & x=\pm1.
\end{cases}
\end{equation*}
The exact solution is
\begin{equation}\label{eq:example2_exact}
u_{\rm ex}(x,y)=\cos(\pi x)\cos(2\pi y).
\end{equation}

Let $D_4$ denote the orthogonal symmetry group of the square, and define
\[
I_{2}=
\begin{pmatrix}
1&0\\
0&1
\end{pmatrix},\qquad
Q_x=
\begin{pmatrix}
-1&0\\
0&1
\end{pmatrix},
\qquad
Q_y=
\begin{pmatrix}
1&0\\
0&-1
\end{pmatrix},
\qquad
Q_xQ_{y}=
\begin{pmatrix}
-1&0\\
0&-1
\end{pmatrix},
\]
Since $A=I_2$, $\widehat{\bb}={\bf 0}$, $c=1$, and $\beta=2$,
\begin{equation}
\Sym(A)=\Sym(\widehat{\bb})=\Sym(c)=\Sym(\beta)=D_4,
\qquad
\Sym(\LL)=\Sym(\Bn)=D_4.
\end{equation}
The volume and boundary sources possess only the two coordinate-reflection symmetries and their product:
\begin{equation}
\Sym(f)=\Sym(h)=\{I_2,Q_x,Q_y,Q_xQ_y\}.
\end{equation}
Consequently,
\begin{equation}\label{commonu}
G_{\rm com}
=
\{I_2,Q_x,Q_y,Q_xQ_y\}.
\end{equation}
From \eqref{eq:example2_exact} and \eqref{commonu},
\begin{equation}
\Sym(u_{\rm ex})
=
\{I_2,Q_x,Q_y,Q_xQ_y\}
=
G_{\rm com},
\end{equation}
which directly verifies Theorem~\ref{thm:solution_symmetry} in this example.

The two reflections reduce the full square to the quarter-domain $\Omega_{++}=(0,1)^2$. The physical edges retain the original Robin condition, whereas the artificial symmetry boundaries satisfy
\begin{equation}
\frac{\partial u}{\partial x}=0
\quad\text{on }x=0,
\qquad
\frac{\partial u}{\partial y}=0
\quad\text{on }y=0.
\end{equation}

\paragraph{Example 3: numerical test with anisotropic variable coefficients and nonzero data.}\leavevmode\par
\noindent Let $\Omega=(-1,1)^2$ and consider
\begin{subequations}\label{eq:variable_num}
\begin{align}
-\nabla\cdot(A\nabla u)
+\widehat{\bb}\cdot\nabla u
+cu
&=f,
&&\text{in }\Omega,\\
(A\nabla u)\cdot\nn+\beta u
&=h,
&&\text{on }\partial\Omega.
\end{align}
\end{subequations}
The coefficients are
\begin{equation*}
A(x,y)=
\begin{pmatrix}
2+x^2 & xy\\
xy & 3+y^2
\end{pmatrix},
\qquad
\widehat{\bb}(x,y)=
\begin{pmatrix}
x(1+y^2)\\
y(1+x^2)
\end{pmatrix},
\end{equation*}
and
\begin{equation*}
c(x,y)=10+x^2+y^2,
\qquad
\beta(x,y)=2+x^2+y^2.
\end{equation*}
Choose the nonzero source data
\begin{equation*}
f(x,y)=2+x^2+y^2+\cos(\pi x)\cos(\pi y),
\qquad
h(x,y)=1+x^2+y^2.
\end{equation*}
No exact solution is prescribed.

The tensor $A$ is symmetric and uniformly positive definite, since
\begin{equation*}
\det A
=
(2+x^2)(3+y^2)-x^2y^2
=
6+3x^2+2y^2
>0.
\end{equation*}

Let
\[
Q_x=
\begin{pmatrix}
-1&0\\
0&1
\end{pmatrix},
\qquad
Q_y=
\begin{pmatrix}
1&0\\
0&-1
\end{pmatrix},
\]
and set
\begin{equation}
K=\{I_2,Q_x,Q_y, Q_xQ_y\}.
\end{equation}
The square has orthogonal symmetry group $D_4$. From the transformation laws, $\Sym(A)=K$, whereas
\begin{equation*}
\Sym(\widehat{\bb})=\Sym(c)=
\Sym(\beta)=\Sym(f)=\Sym(h)=D_4.
\end{equation*}
Consequently,
\begin{equation*}
\Sym(\LL)=\Sym(\Bn)=G_{\rm com}=K.
\end{equation*}
Theorem~\ref{thm:solution_symmetry} therefore predicts reflection symmetry with respect to both coordinate axes.

The full problem can be reduced to the quarter-domain $\Omega_{++}=(0,1)^2$. The original Robin condition is imposed on the inherited physical boundaries $x=1$ and $y=1$, while the artificial symmetry boundaries satisfy
\begin{equation*}
(A\nabla u)\cdot\nn_S=0
\qquad
\text{on }x=0\ \text{and}\ y=0.
\end{equation*}

Let $u_h^{\rm full}$ denote the linear element solution on $\Omega$, and let $\mathcal E u_h^{\rm red}$ denote the solution on $\Omega_{++}$ extended to the full square by the two reflections. Define
\begin{equation*}
e_0=
\frac{
\|u_h^{\rm full}-\mathcal E u_h^{\rm red}\|_{L^2(\Omega)}
}{
\|u_h^{\rm full}\|_{L^2(\Omega)}
},
\qquad
e_1=
\frac{
\|u_h^{\rm full}-\mathcal E u_h^{\rm red}\|_{H^1(\Omega)}
}{
\|u_h^{\rm full}\|_{H^1(\Omega)}
}.
\end{equation*}
The two symmetry defects are
\begin{equation*}
\delta_x=
\frac{
\|u_h^{\rm full}-T_{Q_x}^{(0)}u_h^{\rm full}\|_{L^2(\Omega)}
}{
\|u_h^{\rm full}\|_{L^2(\Omega)}
},
\qquad
\delta_y=
\frac{
\|u_h^{\rm full}-T_{Q_y}^{(0)}u_h^{\rm full}\|_{L^2(\Omega)}
}{
\|u_h^{\rm full}\|_{L^2(\Omega)}
}.
\end{equation*}

\begin{figure}[htbp]
\centering
\begin{subfigure}[t]{0.48\linewidth}
    \centering
    \includegraphics[width=\linewidth]{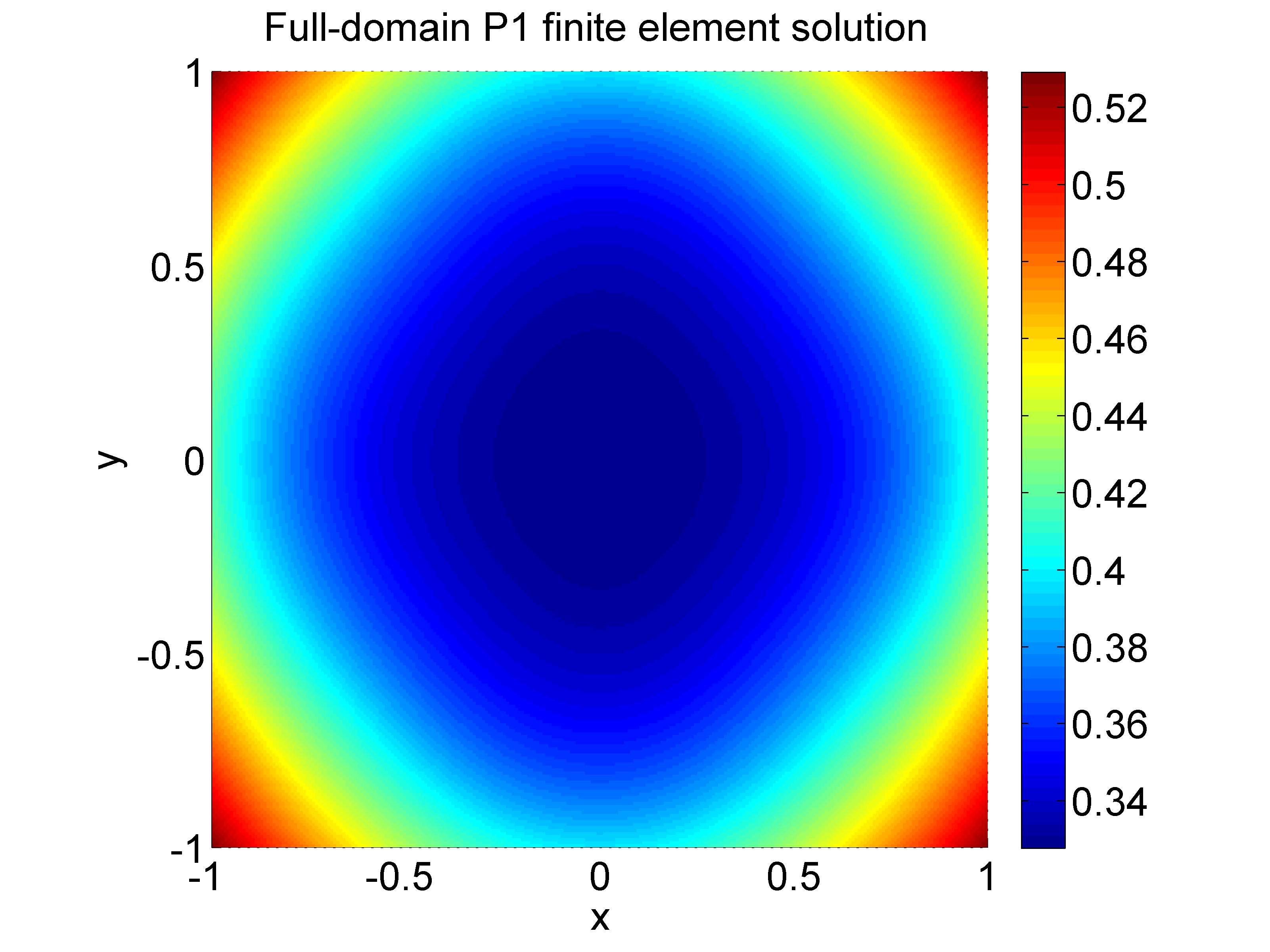}
    \caption{Full-domain solution.}
    \label{fig:example4_full}
\end{subfigure}
\hfill
\begin{subfigure}[t]{0.48\linewidth}
    \centering
    \includegraphics[width=\linewidth]{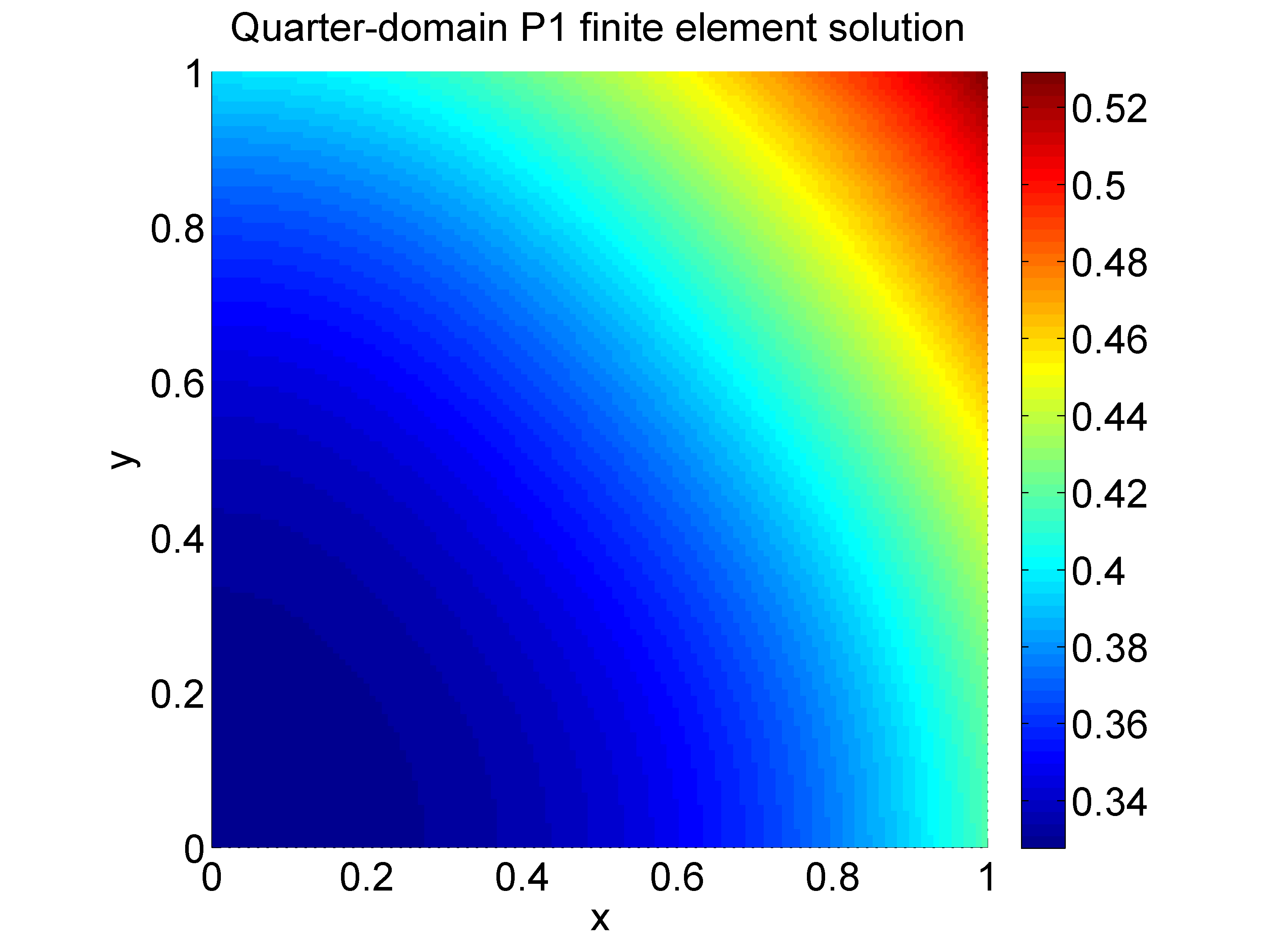}
    \caption{Quarter-domain solution.}
    \label{fig:example4_quad}
\end{subfigure}

\caption{Full-domain and reduced-domain linear finite element solutions for Example~3.}
\label{fig:example4_solution}
\end{figure}

Figure~\ref{fig:example4_solution} shows that the quarter-domain solution reproduces the corresponding part of the full-domain solution, with the same spatial distribution and amplitude. This agreement is consistent with the two reflection symmetries predicted by Theorem~\ref{thm:solution_symmetry} and confirms that the artificial generalized Neumann conditions preserve the physical solution on the reduced domain.

\begin{table}[htbp]
\centering
\caption{Comparison between the full-domain and quarter-domain computations for Example~3.}
\label{tab:example4}
\begin{tabular}{ccccccc}
\hline
Mesh
& $N_{\rm full}$
& $N_{\rm red}$
& $\delta_x$
& $\delta_y$
& $e_0$
& $e_1$
\\
\hline
1 & 1089 & 289& 8.076e-16 & 5.904e-16 & 1.531e-15 & 6.611e-15 \\
2 & 4225 & 1089 & 1.405e-15 & 1.408e-15 & 2.634e-15  & 1.334e-14 \\
3 & 16641& 4225 & 2.438e-15 & 2.121e-15 & 4.250e-15 & 2.677e-14 \\
4 & 66049 & 16641 & 9.285e-15 & 9.522e-15 & 1.479e-14 & 6.671e-14 \\
\hline
\end{tabular}
\end{table}

Table~\ref{tab:example4} gives a quantitative verification of the symmetry reduction. As the mesh is refined, the ratio $N_{\rm full}/N_{\rm red}$ approaches $4$, as expected for a reduction from the full square to one quarter of the domain. The symmetry defects $\delta_x$ and $\delta_y$ remain between $10^{-16}$ and $10^{-14}$, showing that the discrete full-domain solution preserves both reflection symmetries to essentially roundoff accuracy. Likewise, $e_0$ and $e_1$ remain of order $10^{-15}$--$10^{-14}$ and $10^{-14}$--$10^{-13}$, respectively, which demonstrates that the reflected reduced-domain solution and the full-domain solution are numerically indistinguishable. The slightly larger values of $e_1$ are expected because the $H^1$ norm also involves derivatives and is therefore more sensitive to floating-point errors.

This example tests the complete symmetry framework with an anisotropic principal tensor, a nonzero first-order coefficient, variable zeroth-order and Robin coefficients, and nonzero volume and boundary sources. The comparison between the full-domain and quarter-domain numerical solutions provides a direct numerical verification of the symmetry-inheritance theorem and the reduced-domain formulation in a genuinely variable-coefficient setting.

\section{Conclusions and Outlook}
The results show that the symmetry of a Robin boundary value problem can be characterized consistently through the transformation properties of its volume coefficients, boundary data, and source terms. In particular, the Robin condition does not alter the mechanism by which reflection symmetry generates an exact artificial boundary condition: the reduced problem retains the original Robin condition on the physical boundary, while a homogeneous generalized Neumann condition arises naturally on the symmetry interface. This observation provides a direct theoretical basis for symmetry-based domain reduction, and the numerical results confirm that the reduced formulations reproduce the corresponding full-domain solutions.

Future work will address antisymmetric solutions, for which reflection symmetry leads to artificial Dirichlet conditions, and will seek a unified treatment of symmetric and antisymmetric modes. Extensions to vector-valued systems, particularly Maxwell equations with anisotropic constitutive parameters and impedance boundary conditions, are also of interest for symmetry-based reduction in computational electromagnetics.

\end{document}